\documentclass[12pt,reqno, a4paper, final]{amsart}
	
\title{Applications of F{\o}lner's condition to quantum groups}
		
\author{David Kyed} 
\address{David Kyed,
Mathematisches Institut,
Georg-Au\-gust-Uni\-versi\-t{\"a}t G{\"o}t\-ting\-en,
Bunsenstra{\ss}e 3-5,
D-37073 G{\"o}ttingen, 
Germany.}
\email{kyed@uni-math.gwdg.de}
\urladdr{www.uni-math.gwdg.de/kyed}

\author{Andreas Thom}
\address{Andreas Thom, Mathematisches Institut der Universit\"at Leipzig,
Johannisgasse 26, 04103 Leipzig, Germany.}
\email{thom@math.uni-leipzig.de}
\urladdr{http://www.math.uni-leipzig/MI/thom}

\thanks{The work of D.K. was funded partially by the DFG grant SCHI 525/7-1 and partially by The Danish Council for Independent Research $\mid$ Natural Sciences.}
\subjclass[2000]{16W30,43A07, 46L89}
\keywords{Quantum groups, amenability, Ore rings}

\usepackage{amsmath}                    
\usepackage{amsthm}
\usepackage{hyperref}
\usepackage{a4wide}										 
\usepackage{mathrsfs}									  
\usepackage[all]{xy}										
\usepackage{url}												
\usepackage{verbatim,epsfig}

\theoremstyle{plain}

\newtheorem{thm}{Theorem}[section]
\newtheorem{cor}[thm]{Corollary}
\newtheorem{lem}[thm]{Lemma}
\newtheorem{defi}[thm]{Definition}
\newtheorem{prop}[thm]{Proposition}

\theoremstyle{definition}

\newtheorem{ex}[thm]{Example}
\newtheorem{rem}[thm]{Remark}

\newcommand{\NN}{{\mathbb N}}
\newcommand{\ZZ}{{\mathbb Z}}

\newcommand{\CC}{{\mathbb C}}

\newcommand{\MM}{{\mathbb M}}
\newcommand{\GG}{{\mathbb G}}

\renewcommand{\L}{{\mathscr L}}

\renewcommand{\max}{{\operatorname{max}}}

\newcommand{\varps}{{\varepsilon}}
\newcommand{\rg}{{\operatorname{rg\hspace{0.04cm}}}}
\newcommand{\htens}{\bar{\otimes}}

\newcommand{\tens}{\otimes}

\newcommand{\spann}{{\operatorname{span}}}

\newcommand{\To}{\longrightarrow}
\newcommand{\supp}{{\operatorname{supp}}}
\newcommand{\red}{{\operatorname{red}}}

\newcommand{\Mor}{\operatorname{Mor}}

\newcommand{\id}{\operatorname{id}}

\newcommand{\del}{{\partial}}

\newcommand{\alg}{{\operatorname{alg}}}

\newcommand{\Irred}{\operatorname{Irred}}

\newcommand{\tenrep}{\mbox{ $\mbox{\scriptsize \sf T}
\hspace{-1.77ex}\bigcirc$}} 
\renewcommand{\int}{{\operatorname{int}}}

\newcommand{\Tr}{\operatorname{Tr}}

\newcommand{\Pol}{{\operatorname{Pol}}}
\newcommand{\sym}{{\operatorname{sym}}}

\begin{document}

\begin{abstract}
Using the F{\o}lner condition for coamenable quantum groups we derive information about the ring theoretical structure of the Hopf algebras arising from such quantum groups, as well as an approximation result concerning the Murray von Neumann dimension associated with the corresponding the von Neumann algebra.

\end{abstract}

\maketitle

\section{Preliminaries on quantum groups}\label{pre-section}
Consider a compact quantum group $\GG$ in the sense of Woronowicz \cite{wor-cp-qgrps}; i.e.~$\GG$ consists of a (not necessarily commutative) unital, separable $C^*$-algebra $C(\GG)$ together with a unital and coassociative  $*$-homomorphism 
$\Delta_\GG\colon C(\GG)\to C(\GG)\tens C(\GG)$ which furthermore has to satisfy a certain non-degeneracy condition \cite{wor-cp-qgrps}.  We remind the reader that such a $C^*$-algebraic compact quantum group automatically gives rise to a purely algebraic quantum group (i.e.~a Hopf $*$-algebra \cite{klimyk}) whose underlying algebra will be denoted $\Pol(\GG)$, as well as a von Neumann algebraic quantum group (see \cite{kustermans-vaes}) whose underlying algebra will be denoted $L^\infty(\GG)$. We also recall that the $C^*$-algebra $C(\GG)$ possesses a distinguished state $h$, called the Haar state, which plays the role corresponding to the Haar measure on a genuine compact group. The compact quantum group is said to be of Kac type if its Haar state is a trace.	Performing the GNS construction with respect to $h$ yields a Hilbert space denoted $L^2(\GG)$ on which $C(\GG)$ acts via the corresponding GNS-representation $\lambda$. The image $\lambda(C(\GG))$ will be denoted $C(\GG)_\red$ and the von Neumann algebra $L^\infty(\GG)$ is by definition the weak operator closure of $C(\GG)_\red$. We denote by $J$ the modular conjugation arising from $h$ and by $\rho\colon L^\infty(\GG)\to B(L^2(\GG)) $ the $*$-anti-homomorphism $\rho(a)=J\lambda(a)^* J$. There is also a universal representation associated with $\GG$ and we will denote the norm closure of $\Pol(\GG)$ in this representation by $C(\GG)_\max$; the reader is referred to \cite{murphy-tuset} for the details on this construction. 

\begin{ex}\	
The canonical example of a compact quantum group, upon which the general definition is modeled,  is obtained by considering a compact, second countable, Hausdorff topological group $G$ and  its $C^*$-algebra $C(G)$ of continuous, complex valued functions. The comultiplication is then the Gelfand dual of the multiplication map and the Haar state is given by integration against the Haar probability measure $\mu$. In this case the von Neumann algebra is $L^\infty(G,\mu)$ and the associated Hopf $*$-algebra is the algebra generated by matrix coefficients arising from the irreducible representations of $G$. 
\end{ex}

\begin{ex}\label{discrete-grp}
For a discrete countable group $\Gamma$ its reduced $C^*$-algebra $C^*_\red(\Gamma)$ can be turned into a compact quantum group by defining the comultiplication on a group element $\gamma \in \Gamma$ by $\Delta\gamma=\gamma\tens\gamma$. In this situation, the Haar state is the standard trace on $C^*_\red(\Gamma)$ and the Hopf-algebra and von Neumann algebra are, respectively, the complex group algebra $\CC\Gamma$ and the group von Neumann algebra $\L(\Gamma)$. 
\end{ex}

To any  quantum group $\GG$ (compact as well as just locally compact) a so-called multiplicative unitary $W$  on $L^2(\GG)\htens L^2(\GG)$ is associated; this is a unitary which (inter alia) has the property that
\begin{align*}
C(\GG)_\red =[(\id\tens \omega)W\mid \omega\in B(L^2(\GG))_*], 
\end{align*}
where, for a subset $X$ of a normed space, $[X]$ denotes the norm closure of the linear space spanned by $X$. Furthermore, a compact quantum group $\GG$ comes with  a dual quantum group $\hat{\GG}$ of so-called discrete type  whose underlying $C^*$-algebra is given by
\begin{align}\label{dual}
c_0(\hat{\GG})&:=[(\omega\tens \id)W\mid \omega\in B(L^2(\GG))_*].
\end{align}
For a detailed treatment of $C^*$-algebraic (locally compact) quantum groups and their duality theory we refer the reader to the work of Kustermans and Vaes  \cite{kustermans-vaes-C*-lc}.\\

The fundamental notions and results from the representation theory of compact groups (e.g.~irreducibility, decomposition into irreducibles, the Peter-Weyl theorem etc.) have counterparts in the (co)representation theory of compact quantum groups. We shall not elaborate further on these results but refer the reader to \cite{woronowicz} and \cite{vandaele} for more details. We will, however, need some notation concerning the corepresentations of $\GG$ which will be set up in the following.  Let $\Irred(\GG)$ denote the set of equivalence classes of irreducible corepresentations; we label this set by an auxiliary (countable) set $J$ and choose for each $\alpha\in J$ a unitary representative $u^\alpha\in C(\GG)\tens B(H_\alpha)$. Abusing notation slightly, we shall often identify $u^{\alpha}$ with the corresponding class in $\Irred(\GG)$. Moreover, we choose a fixed orthonormal basis for $H_\alpha$ and may therefore also regard $u^\alpha$ as an element in $\MM_{n_\alpha}(C(\GG))$ where $n_\alpha=\dim_\CC H_\alpha$.  We remind the reader that the free $\ZZ$-module $R(\GG)=\ZZ[\Irred(\GG)]$ becomes a fusion algebra, in the sense of  \cite{izumi}, in which the product ``$\tenrep$'' is obtained from the tensor product of corepresentations. The dimension function associated with this fusion algebra maps an irreducible, unitary corepresentation $u$ to its matrix size $n_u$ and the conjugation operator $u\mapsto \bar{u}$ is (essentially) given by taking the contragredient corepresentation. We refer to  \cite[Example 2.3]{coamenable-betti}  for a detailed description of the fusion algebra structure. Finally, we remind the reader that a compact quantum group $\GG$ is called coamenable if the counit $\varps\colon \Pol(\GG)\to\CC$ extends to a character on $C(\GG)_\red$. This notion was investigated in detail by B{\'e}dos, Murphy and Tuset in \cite{murphy-tuset}.\\

The aim of the present paper is to derive information of  ring-theoretical nature about coamenable quantum groups from the following result.

\begin{thm}[\cite{coamenable-betti}]
A compact quantum group $\GG$ is coamenable if and only if it satisfies F{\o}lner's condition; i.e.~for any $\varps>0$ and any finite, non-empty subset $S\subseteq \Irred(\GG)$ there exists a finite subset $F\subseteq \Irred(\GG)$ such that
\[
\sum_{u\in \del^{\sym}_S(F)} n_u^2<\varps\sum_{u\in F} n_u^2.
\]
\end{thm}
\noindent Here the symmetric boundary $\del^{\sym}_S(F)$ is defined as follows: The interior and the boundary of $F$ relative to $S$ are defined, respectively, as
\[
\int_S(F)=\{u\in F\mid \forall  v\in S:\supp(u\tenrep v)\subseteq F\} \quad \text{and} \quad \del_S(F)= F\setminus \int_S(F),
\]
and the symmetric boundary is then given by
\[
\del_S^\sym(F)= \del_S(F)\cup \del_S(F^c).
\]
Here $F^c$ denotes the complement in $\Irred(\GG)$ of the set $F$ and $u\tenrep v$ denotes the product in the fusion algebra $\ZZ[\Irred(\GG)]$. Note that if $F,S$ and $\varps$ are as in the F{\o}lner condition, we trivially get that
\begin{align}\label{ineq}
\sum_{u\in \del_S(F)} n_u^2<\varps\sum_{u\in F} n_u^2. 
\end{align}
Simplifying notation, we write $|F|=\sum_{u\in F}n_u^2$ for a finite subset $F\subseteq \Irred(\GG)$ and the inequality (\ref{ineq}) may therefore be written in a more compact form as
\[
|\del_S(F)|<\varps |F|.
\]
In particular, the F{\o}lner condition allows us to choose a sequence of subsets $(F_k)_{k\in \NN}$ of $\Irred(\GG)$ such that 
\[
\frac{|\int_S(F_k)|}{|F_k|}\underset{k\to\infty}{\To}1.
\]
\begin{rem}
Note that quantum groups of the form $C^*_\red(\Gamma)$ are coamenable exactly when $\Gamma$ is amenable, and in this case the quantum F{\o}lner condition identifies with the classical F{\o}lner condition for $\Gamma$. All commutative examples are automatically coamenable since the counit is given by evaluation at the identity in the corresponding compact group and therefore automatically globally defined and bounded. By results of Banica \cite{banica-subfactor}, the $q$-deformed $SU(2)$-groups $SU_q(2)$ of Woronowicz are all coamenable and so is the quantum permutation group on four points $S_4^+$ (the latter is furthermore of Kac type). Further examples of compact coamenable quantum groups of Kac type can be obtained by forming crossed products of discrete amenable groups acting on compact Kac algebras \cite{canniere} and even more general examples  of compact coamenable quantum groups can be obtained by considering the (cocycle) crossed product of a (cocycle) matched pair of a discrete amenable quantum group and a coamenable compact quantum group \cite{vaes-bicrossed-amenability, tomatsu-amenable,vaes-vainerman}. In connection with Theorem \ref{ore-thm} we note that the liberated orthogonal group $O_2^+\simeq SU_{-1}(2)$ is coamenable and of Kac type and that its associated Hopf algebra is a domain.

\end{rem}

Assume for the rest of this section that $\GG$ is of Kac type; then the discrete dual quantum group $\hat{\GG}$ is unimodular and its underlying Hopf algebra $c_c(\hat{\GG})$ is $*$-isomorphic\footnote{the isomorphism also exists in the non-Kac case.} to
\[
\bigoplus^{\alg}_{\alpha\in J}B(H_\alpha).
\]
The (bi-invariant) Haar functional $\hat{h}\colon c_c(\hat{\GG})\to \CC$ with $\hat{h}(h)=1$ is given by the simple formula
\[
(x_\alpha)_{\alpha\in J}\longmapsto\sum_{\alpha\in J}n_\alpha\Tr_{H_\alpha}(x_\alpha),
\]
where $\Tr_{H_\alpha}(-)$ is the non-normalized trace on $B(H_\alpha)$. For a finite subset $F\subseteq \Irred(\GG)$ we denote by $P_F$ the central projection in $c_c(\hat{\GG})$ given by $\sum_{u\in F} 1_{n_u}$, and we note that $\hat{h}(P_F)=|F|$.  Moreover, we denote by $W_F$ the $|F|$-dimensional subspace
\[
\spann_\CC\{u_{ij}\mid u\in F\}\subseteq \Pol(\GG).
\]
The algebra $c_c(\hat{\GG})$ has a natural representation $L\colon c_c(\hat{\GG})\to B(L^2(\GG))$ and the functional $\hat{h}$ gives rise to a normal, semi-finite, faithful Haar weight on the enveloping von Neumann algebra $\ell^\infty(\hat{\GG}):=L(c_c(\hat{\GG}))''$ turning it into a Kac algebra of discrete type. In this representation, the projection $P_E$ projects onto the finite dimensional subspace $W_{\bar{E}}$ where $\bar{u}\in \Irred(\GG)$ denotes the conjugate of $u$. Clearly the functional $h$ gives rise to a faithful (tracial) state on the enveloping von Neumann algebra $L^\infty(\GG)=\lambda(\Pol(\GG))''$ and the two Haar functionals are linked via the following simple relation.

\begin{prop}[\cite{vaes-van-daele-heisenberg}]\label{trace-formula}
For any $a\in {L^\infty(\GG)}$ and any $x\in \ell^\infty(\hat{\GG})$ with $\hat{h}(x^*x)<\infty$ we have
$
\Tr(a^*x^*x a)=h(a^*a)\hat{h}(x^*x).
$ 
\end{prop}
\noindent Here $\Tr(-)$ denotes the ordinary trace on $B(L^2(\GG))$.

\section{The relative dimension function}\label{reldim-section}
Consider a compact quantum group of Kac type and a finite subset $F\subseteq \Irred(\GG)$ and denote by $P_F^n\in c_c(\hat{\GG})^n$ the diagonal amplification of the central projection $P_F\in c_c(\hat{\GG})$; i.e.~$P_F^n$ is the projection onto the finite dimensional subspace
\[
W_{\bar{F}}^n=\spann_\CC\{ u_{ij}\mid u\in \bar{F}\}^n\subseteq L^2(\GG)^n.
\]
For a closed subspace $K\subseteq L^2(\GG)^n$ we denote by $Q_K$ the orthogonal projection onto $K$ and define the dimension of $K$ relative to $F$ as
\[
\dim_F(K)=|F|^{-1}\Tr_{n}(Q_K P_{\bar{F}}^n),
\]
where $\Tr_n(-)$ denotes the trace on $B(L^2(\GG)^n)$. Here, and in what follows, we suppress the representations $\lambda$ and $L$ for the sake of notational convenience.
\begin{prop}\label{dim-prop}
The relative dimension function $\dim_F(-)$ has the following properties:
\begin{itemize}
\item[(i)] The number $\dim_F(K)$ is non-negative and finite.
\item[(ii)] If $K_1\subseteq K_2$ then $\dim_F(K_1)\leq \dim_F(K_2)$.
\item[(iii)] If  $K$ is ${L^\infty(\GG)}$-invariant then $\dim_F(K)=\dim_{L^\infty(\GG)}(K):=h(Q_K)$; the Murray-von Neumann dimension of $K$. 
\item[(iv)] If $K\subseteq \spann_\CC\{u_{ij}\mid {u}\in F\}^n$ then $\dim_F(K)=|F|^{-1}\dim_\CC(K)$.
\end{itemize}
\end{prop}

\begin{proof}
Properties (i) and (ii) follows from traciality and positivity of the standard trace $\Tr_{n}(-)$, and (iv) is seen through a straight forward calculation using the orthonormal basis $\{\sqrt{n_\alpha}u_{ij}^{\alpha}\mid \alpha\in I, 1\leq i,j\leq n_\alpha \}$ for $L^2(\GG)$. Using the unimodularity of $\hat{\GG}$ and the trace formula in Proposition \ref{trace-formula} one sees that for any matrix $T$ in either $\MM_n(L^\infty(\GG))$ or $\MM_n(L^\infty(\GG)')$ we get
\[
\Tr_{n}(T^*P_F^n T)=h_n(T^*T)\hat{h}(P_F),
\]
where $h_n\colon \MM_n(B(L^2(\GG)))\to \CC$ is given by $h_n((T_{ij})_{i,j=1}^n)=\sum_{i=1}^n h(T_{ii})$. From this formula (iii) follows. See \cite{coamenable-betti} Lemma 5.1 for more details. 
\end{proof}

\begin{rem}
The relative dimension function above is a quantum analogue of a construction considered for groups by Eckmann in \cite{eckmann} and Elek in \cite{elek-zdc}. 
\end{rem}

\section{Zero-divisors in quantum groups}
In \cite{elek-zdc} Elek proves that for amenable torsion free groups the zero-divisor conjecture of Kaplansky is equivalent to Linnell's analytic zero-divisor conjecture. The aim of this section is to boost Elek's argument to obtain the following result:

\begin{thm}\label{zd-thm}
If $\GG$ is coamenable and  $a\in \Pol(\GG)$ is not a left zero-divisor then $\lambda(a)$ acts with trivial kernel on $L^2(\GG)$. Similarly, if $a$ is not a right zero-divisor in $\Pol(\GG)$ then $\rho(a)$ has trivial kernel.
\end{thm}
As in the previous section, $\GG$ denotes here a compact quantum group of Kac type. We note that when applied to quantum groups of the form $C^*_\red(\Gamma)$, with $\Gamma$ discrete and amenable, Theorem \ref{zd-thm} identifies with Elek's original result. The contents of Theorem \ref{zd-thm} can also be derived from the proof of \cite[Theorem 6.1]{coamenable-betti}, but we give here the following shorter and more illuminating proof.

\begin{proof}
By symmetry, it suffices to treat the case where $a$ is not a right zero-divisor.
Since $\{u_{ij}^\alpha \mid \alpha\in J\}$ constitutes a linear basis for the space $\Pol(\GG)$ the element $a\in \Pol(\GG)$ has a unique linear expansion $a=\sum_{i,j,\alpha}t_{ij}^\alpha u_{ij}^\alpha$ and we may therefore consider its \emph{support} which is defined as
\[
S=\supp(a)=\{u^\alpha \in \Irred(\GG)\mid  \exists \  i,j\in\{ 1,\dots, n_\alpha \} : t_{ij}^\alpha\neq 0 \}. 
\]
Since $\GG$ is assumed to be coamenable, we can choose (see Section \ref{reldim-section}) a sequence of finite sets $F_k\subseteq \Irred(\GG)$ such that
\[
\frac{|\int_SF_k|}{|F_k|}\underset{k\to\infty}{\To}1.
\]
Denote $\int_S(F_k)$ by $G_k$ for simplicity. Because any product of matrix coefficients $u_{ij}v_{kl}$ is contained in the linear span of the matrix coefficients of the tensor product $u\tenrep v$, we see that $\rho(a)$ restricts to an operator
\[
\rho(a)_k\colon W_{G_k}\To W_{F_k}.
\]
We now prove that
\begin{align*}
\dim_{F_k}(\ker(\rho(a)_k))\underset{k\to\infty}{\To} \dim_{L^\infty(\GG)}(\ker(\rho(a))).\tag{$\dagger$}
\end{align*}
\begin{proof}[Proof of $(\dagger)$]
Since $\rho(a)\in {L^\infty(\GG)'}$ both $\ker(\rho(a))$ and $\overline{\rg(\rho(a))}$ are closed ${L^\infty(\GG)}$-invariant subspaces. Using Proposition \ref{dim-prop}  we therefore get
\begin{align}\label{eq1}
|F_k|^{-1}\dim_\CC(\ker(\rho(a)_k))&=\dim_{F_k}(\ker(\rho(a)_k))\notag\\
&\leq \dim_{F_k}(\ker(\rho(a)))\notag\\
&=\dim_{L^\infty(\GG)}(\ker(\rho(a)));
\end{align}
and 
\begin{align}\label{eq2}
|{F_k}|^{-1}\dim_\CC(\rg(\rho(a)_k))&=\dim_{F_k}(\rg(\rho(a)_k))\notag\\
&\leq \dim_{F_k}(\overline{\rg(\rho(a))})\notag\\
&=\dim_{L^\infty(\GG)}(\overline{\rg(\rho(a))}).
\end{align}
Since the Murray-von Neumann dimension is additive, the inequalities (\ref{eq1}) and (\ref{eq2}) yield
\begin{align*}
1&=\dim_{L^\infty(\GG)}(L^2(\GG))\\
&=\dim_{L^\infty(\GG)}(\ker(\rho(a)))+\dim_{L^\infty(\GG)}(\overline{\rg(\rho(a))})\\
&\geq \dim_{F_k}(\ker(\rho(a)_k))+\dim_{F_k}(\rg(\rho(a)_k))\\
&\geq |F_k|^{-1}\dim_\CC(\ker(\rho(a)_k))+ |F_k|^{-1}\dim_\CC(\rg(\rho(a)_k))\\
&=|F_k|^{-1}\dim_{\CC}W_{G_k}\\
&=|F_k|^{-1}|G_k| .
\end{align*}
Since $|F_k|^{-1}|G_k|\underset{k\to \infty}{\To} 1$ this forces
\begin{align*}
\lim_{k\to\infty}\dim_{F_k}(\ker(\rho(a))_k)&=\dim_{L^\infty(\GG)}(\ker(\rho(a))); \\
\lim_{k\to\infty}\dim_{F_k}(\rg(\rho(a))_k) &=\dim_{L^\infty(\GG)}(\overline{\rg(\rho(a))}),
\end{align*}
as desired.
\end{proof}
\noindent Theorem \ref{zd-thm} now follows easily: If $\ker(\rho(a))$ is non-trivial then $\dim_{L^\infty(\GG)}(\ker(\rho(a)))>0$ and by $(\dagger)$ there must exist a $k\in \NN$ such that $\dim_{F_k}\ker(\rho(a)_k)>0$. This can only happen if $\ker(\rho(a)_k)$ is non-trivial and since $\ker(\rho(a)_k)\subseteq W_{G_k}\subseteq \Pol(\GG)$ this proves the claim.
\end{proof}
In particular, Theorem \ref{zd-thm} implies the following stability of regularity:
\begin{cor}\label{regular-cor}
If $a\in \Pol(\GG)$ is a regular element\footnote{i.e.~not a zero-divisor from either side.} then it stays regular in the over-ring $L^\infty(\GG)$
\end{cor}
\begin{proof}
This follows from the fact that the GNS-construction provides an embedding of $L^\infty(\GG)$ into $L^2(\GG)$.
\end{proof}

Our next application regards a non-commutative analogue of the construction of the field of fractions associated with a (commutative) integral domain --- the so-called Ore localization. We first quickly review the basics on this construction and refer the reader to \cite{lam-lectures} for a thorough treatment of the subject.
\begin{defi}
Let $R$ be a ring and let $Z\subseteq R$ be a multiplicative subset containing no zero-divisors. Then $R$ is said to satisfy the (right) Ore condition with respect to $Z$ if for all $a\in R$ and $s\in Z$ there exist $b\in R$ and $t\in S$ such that $at=sb$.
\end{defi}
If $R$ satisfies the Ore condition with respect to $Z$ then there exists a ring $RZ^{-1}$ and a ring homomorphism $\iota\colon R\to RZ^{-1}$ such that  $\iota(s)$ is invertible for any $s\in Z$, each element in $RZ^{-1}$ can be written as $\iota(r)\iota(s)^{-1}$ for $s\in Z$ and $r\in R$  and which furthermore is universal in the following sense: For any ring $S$ and any homomorphism $\varphi\colon R\to S$ such that $\varphi(Z)\subseteq S^{\times}$ there exists $\Phi\colon RZ^{-1}\to S$ such that $\Phi\circ \iota=\varphi$. In particular, if $R$ is a domain and $Z=R\setminus\{0\}$, then $\iota$ is injective and $RZ^{-1}$ is a skew field.

\begin{thm}\label{ore-thm}
If $\GG$ is a coamenable, compact quantum group of Kac type and $\Pol(\GG)$ is a domain then $\Pol(\GG)$ satisfies the (right) Ore condition with respect to the set $Z$ of all non-zero elements. Moreover, the skew field $\Pol(\GG)Z^{-1}$  embeds into the ring of  operators affiliated with $L^\infty(\GG)$.
\end{thm}
The proof is an extension of an old argument due to Tamari \cite{tamari}.
\begin{proof}
Given $a\in \Pol(\GG)$ and $s\in Z$ we must find $b\in \Pol(\GG)$ and $t\in Z$ such that $at=sb$. Denote by $S\subseteq \Irred(\GG)$ the union $\supp(a)\cup\supp(s)$ and choose, according to the F{\o}lner condition, a finite subset $F\subseteq \Irred(\GG)$ such that $|\del_S F|<\frac12 |F|$. We then have
\[
|\del_SF|<\tfrac12|F|=\tfrac12|\int_S F| + \tfrac12|\del_S F|
\]
such that $|\del_S F|<|\int_S F|$. Consider the linear map 
\[
W_{\int_S F}\oplus W_{\int_S F}\ni(x,y)\overset{\alpha}{\longmapsto} ax - sy \in  W_{F},
\]
and note that
\begin{align*}
\dim_\CC W_{F} &= \dim_\CC W_{\int_S F} +\dim_\CC W_{\del_S F}\\
&=|\int_S F|+|\del_S F|\\
&<|\int_SF|+|\int_SF|\\
&=\dim_\CC(W_{\int_S F}\oplus W_{\int_S F}).
\end{align*}
We may therefore choose a non-trivial element $(t,b)\in \ker(\alpha)$. Note that this pair will solve the desired equation. We therefore just have to prove that $t$ is non-zero. But if $t=0$ then $0=at=sb$ and since $s\neq 0$ and $\Pol(\GG)$ is a domain this forces $b=0$, contradicting the choice of $(t,b)$. Thus, $\Pol(\GG)$ satisfies the Ore condition with respect to $Z$. Also $L^\infty(\GG)$ is an Ore ring with respect to its regular elements and its Ore localization identifies with its algebra of affiliated operators \cite{reich01}. From Corollary \ref{regular-cor} it follows that each non-zero element in $\Pol(\GG)$ stays regular in $L^\infty(\GG)$ and hence becomes invertible considered as an affiliated operator. By universality of the Ore localization we get an embedding of the skew field $\Pol(\GG)Z^{-1}$ into the algebra of affiliated operators.

\end{proof}

\begin{rem}
Note that the preceding theorem clarifies the possibility of extending the Atiyah conjecture to a context of Kac algebras. Indeed, under suitable assumptions regarding the torsion in the Kac algebra, one can conjecture that there exists an embedding of the Kac algebra into a skew field inside the algebra of affiliated operators. We show that in the coamenable case, as for group rings of amenable groups, being a domain already implies the existence of such a skew field inside the algebra of affiliated operators.
\end{rem}

\section{An approximation result}

Let $\Gamma$ be a discrete group and let $(\Gamma_i)_{i\in I}$ be an inverse family of normal subgroups directed by inclusion. Denote by $\pi_i$ the canonical surjection $\Gamma\to \Gamma/\Gamma_i$ as well as its natural extensions $\CC[\Gamma]\to \CC[\Gamma/\Gamma_i]$ and $\MM_n(\CC[\Gamma])\to\MM_n(\CC[\Gamma/\Gamma_i])$. Each matrix $T\in \MM_n(\CC[\Gamma])$ gives, by right multiplication, rise to a $\Gamma$-equivariant operator $R_T\colon \CC[\Gamma]^n\to\CC[\Gamma]^n$ 
which extends to an $\L(\Gamma)$-equivariant operator $R_T^{(2)}\in B(\ell^2(\Gamma)^n)$. We denote the matrix $\pi_i(T)\in \MM_n(\CC[\Gamma/\Gamma_i])$ by $T_i$ for notational convenience. A standing approximation conjecture, due to L{\"u}ck \cite{Luck02}, in the theory of $L^2$-invariants states that if $\cap_{i\in I} \Gamma_i=\{e\}$ then 
\[
\dim_{\L(\Gamma)}\ker(R_T^{(2)})=\lim_{i}\dim_{\L(\Gamma/\Gamma_i)}\ker(R_{T_i}^{(2)}).
\]
The approximation conjecture has been verified for matrices with entries in the integral group ring by L\"uck. Unfortunately, there exists no analogue for the \emph{integral} group ring in the theory of Kac algebras. However, for matrices with entries in the complex group ring the conjecture has also been verified in many cases; for torsion free, elementary amenable groups it was verified by Dodziuk, Linnell, Mathai, Schick and Yates in \cite{dodziuk-et-al}, and later Elek verified it for all amenable groups in \cite{elek-approximation-conj}. We now wish to consider a quantum analogue of this result. For this, let $\GG$ be a compact quantum group of Kac type and consider a projective system of compact quantum subgroups $(\GG_i)_{i\in I}$. That is; for each $i,j\in I$ with $j\geq i$ we have surjective $*$-homomorphisms $\pi_i\colon C(\GG)_\max\to C(\GG_i)_\max$ and $\pi_{ij}\colon C(\GG_j)_\max\to C(\GG_i)_\max$ that are compatible with the comultiplications and make the following diagram commutative.
\[
\xymatrix{& C(\GG)_\max \ar@{->>}[dl]_{\pi_j} \ar@{->>}[dr]^{\pi_i} &  \\
C(\GG_j)_\max \ar@{->>}[rr]_{\pi_{ij}}  & & C(\GG_i)_\max }
\]
Each corepresentation $u\in \MM_n(C(\GG))$ gives rise to a corepresentation $\pi_i(u)=(\pi_i(u_{kl}))_{kl}\in \MM_n(C(\GG_i))$ which need not be irreducible even when $u$ is so. But in case $\pi_i(u)$ is irreducible for some $i\in I$ it has to stay irreducible; i.e.~if $j\geq i$ then $\pi_j(u)$ is also irreducible. As in the group case, each matrix $T\in \MM_n(\Pol(\GG))$ gives rise to an $L^\infty(\GG)$-invariant operator $R_T^{(2)}\in B(L^2(\GG))$ and we denote by $T_i$ the matrix $\pi_i(T)\in \MM_n(\Pol(\GG_i))$. 

\begin{defi}\label{loc-inj}
The induced map on the level of fusion algebras $\pi_i\colon R(\GG)\to R(\GG_i)$  is called injective on $F\subseteq \Irred(\GG)$ (or just $F$-injective) if $\pi_i$ maps $F$ injectively into $\Irred(\GG_i)$. The family $(\pi_i)_{i\in I}$ is called locally injective if for all finite $F\subseteq \Irred(\GG)$ there exists $i\in I$ such that $\pi_i$ is $F$-injective.

\end{defi}

We remark that neither the injectivity nor the fact that $\pi_i(F)\subseteq \Irred(\GG_i)$ is  automatic. Also note that if $\pi_i$ is injective on $F$ for some $i$ then $\pi_{j}$ is injective on $F$ whenever $j\geq i$. Another consequence of local injectivity is that the Haar state of $\GG$ can be approximated; more precisely the following holds.

\begin{prop}
The family $(\pi_i)_{i\in I}$ is locally injective if and only if  $h_\GG(a)=\lim_i h_{\GG_i}(\pi_i(a))$ for any $a\in \Pol(\GG)$.
\end{prop}
\begin{proof} 
Assume first that $(\pi_i)_{i\in I}$ is locally injective. If $a\in \Pol(\GG)$ has linear expansion $a=\sum_{\alpha,k,l} t_{kl}^{\alpha}u_{kl}^\alpha$ then $h(a)$ is just the coefficient of the trivial corepresentation. Now choose $i_0\in I$ such that $\pi_{i_0}$ is injective on $\supp(a)$; then for all $i\geq i_0$ we have that $\pi_i(a)=\sum_{\alpha,k,l}t_{kl}^{\alpha}\pi_i(u_{kl}^\alpha)$ and that
$
\{\pi_i(u_{kl}^\alpha)\mid u^\alpha\in \supp(a)\}
$
is a linearly independent set of vectors  in $\Pol(\GG_i)$. Hence $h_\GG(a)=h_{\GG_i}(\pi_i(a))$ for all $i\geq i_0$.  Assume conversely that the approximation property holds and consider some finite $F\subseteq \Irred(\GG)$. For an irreducible, unitary corepresentation $u$ we have that $h_\GG(\chi(u\tenrep \bar{u}))=1$ (where $\chi$ is the character map \cite{woronowicz-pseudo}) and hence
\begin{align*}
1 &=h_\GG(\chi(u\tenrep \bar{u}))\\
&=\lim_i h_{\GG_i}(\pi_i(\chi(u\tenrep \bar{u})))\\
&=\lim_i h_{\GG_i}(\chi(\pi_i(u)\tenrep \overline{\pi_i(u)}))\\
&=\lim_i\dim_{\CC}\Mor(e,\pi_i(u)\tenrep \overline{\pi_i(u)}).
\end{align*}
Therefore $\dim_{\CC}\Mor(e,\pi_i(u)\tenrep \overline{\pi_i(u)})=1$ eventually and hence there exists $i_1\in I$ such that $\pi_i(F)\subseteq \Irred(\GG_i)$ for $i\geq i_1$. Similarly, if $u,v\in F$ are different we have $h_\GG(\chi(u\tenrep \bar{v}))=0$ and as above the approximation property implies  that $h_{\GG_i}(\chi(\pi_i(u)\tenrep \overline{\pi_i(v)}))=0$ eventually. Thus, there exists $i_2\in I$ such that $\pi_i(F)$ consists of inequivalent corepresentations when $i\geq i_2$. So, if $i\geq \max\{i_1,i_2\}$ we have that $\pi_i(F)$ consists of inequivalent, irreducible corepresentations as desired.

\end{proof}

Note that the proof actually gives the stronger statement that $h_{\GG_i}(\pi_i(a))=h_\GG(a)$ eventually. The aim of this section is to prove the following.

\begin{thm}\label{quantum-pape}
If $\GG$ is coamenable and of Kac type and if the family $(\pi_i)_{i\in I}$ is locally injective then
\[
\dim_{L^\infty(\GG)}\ker(R_T^{(2)})=\lim_i\dim_{L^\infty(\GG_i)}\ker(R_{T_i})
\] 
for each $T\in \MM_n(\Pol(\GG))$.
\end{thm}
The verification of the approximation conjecture for amenable groups was recently simplified by Pape in \cite{pape} and the proof of Theorem \ref{quantum-pape} follows the outline of his proof.

\begin{rem}
Since $\GG$ is assumed  coamenable the same is true for each of the $\GG_i$'s. This follows by noting that the surjection $\pi_i\colon C(\GG)_\max\to C(\GG_i)_\max$ induces an injection $\iota_i\colon \ell^\infty(\hat{\GG}_i)\to \ell^\infty(\hat{\GG})$ respecting the coproducts. If $\GG$ is coamenable then $\hat{\GG}$ is amenable \cite{tomatsu-amenable} and $\ell^\infty(\hat{\GG})$ therefore allows an invariant mean which restricts to an invariant mean on $\ell^\infty(\hat{\GG}_i)$ proving that $\GG_i$ is coamenable. Similarly, since $\GG$ is assumed to be of Kac type each  $\GG_i$ is also of Kac type \cite[Lemma 2.9]{tomatsu-coideals}; this is of course needed in order for $\dim_{L^\infty(\GG_i)}(-)$ to make sense.
\end{rem}

Before the actual proof we will set up some notation and prove two small lemmas. Since the statement in Theorem \ref{quantum-pape} is trivial if $T=0$ we can assume $T\neq 0$. Each entry $T_{ij}$ has a unique linear expansion $T_{ij}=\sum_{\alpha,p,q} t_{\alpha,p,q}^{ij}u_{pq}^{\alpha}$ and  we now define the \emph{support of $T$} as the set
\[
S=\{u^\alpha\in \Irred(\GG)\mid \exists \ i,j,p,q : t_{\alpha,p,q}^{ij}\neq 0\},
\]
which is non-empty since $T$ is assumed non-zero. For a finite subset $F\subseteq \Irred(\GG)$ we consider, as in the previous section,  the subspace
\[
W_F^n=\spann_\CC\{ u_{ij}^\alpha\mid u^\alpha \in {F}\}^n\subseteq L^2(\GG)^n,
\]
and we see that $R_T^{(2)}$ restricts to an operator $R_T^{F}\colon W_{{\int_S(F)}}^n\To W_F^n$. 
The following lemma is a quantum group analogue of \cite[Lemma 1]{pape}.
\begin{lem}\label{pape-lem}
For each finite $F\subseteq \Irred(\GG)$ we have
\[
0\leq \dim_{L^\infty(\GG)}\ker(R_T^{(2)})-\dim_F\ker(R_T^F)\leq n \frac{|\del_SF|}{|F|}.
\]
\end{lem}
We remind the reader that $|F|$ is defined as $\hat{h}_\GG(P_F)=\sum_{u\in F}n_u^2$. 
\begin{proof}
The first inequality follows from the inclusion $\ker(R_T^F)\subseteq \ker(R_T^{(2)})$. To prove the second inequality, we first note that Proposition \ref{dim-prop} together with the dimension theorem from linear algebra implies that
\begin{align*}
\dim_F\ker(R_T^F)+\dim_F \rg(R_T^F) &= |F|^{-1}\Big{(}\dim_\CC \ker(R_T^F)+\dim_\CC\rg(R_T^F)\Big{)}\\
&=|F|^{-1}\dim_\CC{W_{\int_S(F)}^n}\\
&=|F|^{-1}n(|F|-|\del_SF|)\\
&= n-n|F|^{-1} |\del_SF|. 
\end{align*}
Now note that $\overline{\rg(R_T^{(2)})}$ is $L^\infty(\GG)$-invariant so that
\[
\dim_{L^\infty(\GG)}\overline{\rg(R_T^{(2)})}=\dim_F \overline{\rg(R_T^{(2)})}\geq \dim_F \rg(R_T^F).
\]
Using this and the additivity of $\dim_{L^\infty(\GG)}(-)$ we get the desired inequality:
\begin{align*}
\dim_{L^\infty(\GG)}\ker(R_T^{(2)})-\dim_F\ker(R_T^F) &= n-\dim_{L^\infty(\GG)}\overline{\rg(R_T^{(2)})}-\dim_F\ker(R_T^F)\\
&\leq n-\dim_F\rg(R_T^F)-\dim_F\ker(R_T^F)\\
&=n-\Big{(}n-n|F|^{-1} |\del_SF|\Big{)}\\
&= n|F|^{-1} |\del_SF|.
\end{align*}
\end{proof}
For a finite set $F\subseteq \Irred(\GG)$ the set $\pi_i(F)\subseteq \ZZ[\Irred(\GG_i)]$ will be denoted by $F_i$ in the following. We remark that if $\pi_i$ is injective on $F$ then we have $F_i\subseteq \Irred(\GG_i)$.

\begin{lem}\label{rand-lem}
Let $F\subseteq \Irred(\GG)$ be a finite set and assume that $\pi_i\colon R(\GG)\to R(\GG_i)$ is locally injective (see Definition \ref{loc-inj}) on the set
\[
\Omega= F\cup S\cup \bigcup_{x\in F,s\in S}\supp(x\tenrep s)
\]
for some $i\in I$. Then $S_i$ is the support of $T_i$ and $\pi_i(\del_SF)=\del_{S_i}(F_i)$. Moreover, for any $E\subseteq \Omega$ we have $\hat{h}_{\GG_i}(P_{E_i})=\hat{h}_{\GG}(P_E)$.

\end{lem}
\begin{proof}
The equality $S_i=\supp(T_i)$ follows directly from the injectivity of $\pi_i$ on $S$. The injectivity on $\Omega$ also implies that 
\[
\forall x\in F \ \forall s\in S: \pi_i(\supp(x\tenrep s))=\supp(\pi_i(x)\tenrep \pi_i(s)).
\]
From this the equality $\pi_i(\del_SF)=\del_{S_i}F_i$ follows. Since $\hat{h}_\GG(1_u)=n_u^2$ for any $u\in \Irred(\GG)$ the $\Omega$-injectivity gives
\[
\hat{h}_{\GG}(P_E)=\sum_{u\in E}n_u^2=\sum_{u\in E_i}n_u^{2}=\hat{h}_{\GG_i}(P_{E_i}).
\]

\end{proof}
We now give the proof of Theorem \ref{quantum-pape}.
\begin{proof}[Proof of Theorem \ref{quantum-pape}.]
For given $\varps>0$ the F{\o}lner condition provides an $F\subseteq \Irred(\GG)$ such that
\[
|\del_S(F)|<\frac{\varps}{2n}|F|,
\]
where $S$ is the support of $T$ as defined earlier. By assumption we can find $i_0\in I$ such that  $\pi_i$ is locally injective on the set
\[
\Omega=F\cup S \cup \bigcup_{x\in F, s\in S}\supp(x\tenrep s)
\] 
whenever $i\geq i_0$. In particular $\{\pi_i(\sqrt{n_\alpha}u_{kl}^\alpha)\mid 1\leq k,l\leq  n_\alpha, u^\alpha\in \Omega\}$ is an orthonormal set of vectors in $L^2(\GG_i)$. From Lemma \ref{rand-lem} we therefore obtain a commutative diagram
\[
\xymatrix{
W_{\int_SF}^n \ar[rr]_{\sim}^{\pi_i^{(2)}} \ar[dd]_{R_T^F}  & & W_{\int_{S_i}F_i}^n\ar[dd]^{R_{T_i}^{F_i}} \\
 &              \\
W_F^n\ar[rr]^{\sim}_{\pi_i^{(2)}}  &&  W_{F_i}^n  }
\]
where the horizontal arrows are just given by $\pi_i$, but now considered as a unitary map of finite dimensional Hilbert spaces. From this we get $\dim_F(\ker(R_T^F))=\dim_{F_i}(\ker(R_{T_i}^{F_i}))$ and, using Lemma \ref{pape-lem}, we now get
\begin{align*}
& \hspace{0.56cm}|\dim_{L^\infty(\GG)}\ker(R^{(2)}_T)-\dim_{L^\infty(\GG_i)}\ker(R_{T_i}^{(2)})|\\
&\leq |\dim_{L^\infty(\GG)}\ker(R_T^{(2)})-\dim_F\ker(R_T^F)| + |\dim_F\ker(R_T^{F})-\dim_{L^\infty(\GG_i)}\ker(R_{T_i}^{(2)})|\\ 
&\leq n \frac{|\del_SF|}{|F|} + |\dim_{F_i}\ker(R_{T_i}^{F_i})-\dim_{L^\infty(\GG_i)}\ker(R_{T_i}^{(2)})|\\
&\leq n \frac{|\del_SF|}{|F|} + n \frac{\hat{h}_{\GG_i}(P_{\del_{S_i}F_i})}{\hat{h}_{\GG_i}(P_{F_i})}\\
&= 2n \frac{|\del_SF|}{|F|}.
\end{align*}
Therefore $|\dim_{L^\infty(\GG)}\ker(R^{(2)}_T)-\dim_{L^\infty(\GG_i)}\ker(R_{T_i}^{(2)})|<\varps$ whenever $i\geq i_0$ and the desired convergence follows.
\end{proof}


\def\cprime{$'$}

\end{document}